\documentclass[a4paper,12pt]{amsart}
\usepackage[english]{babel}
\usepackage[T1]{fontenc}
\usepackage{array}
\usepackage{makecell}
\usepackage[a4paper,margin=2.5cm]{geometry}
\usepackage[justification=centering]{caption}

\usepackage{enumerate}
\usepackage{tabularx} 
\usepackage{amsmath}  
\usepackage{graphicx} 

\usepackage{amsmath,amscd,amsbsy,amssymb,latexsym,url,bm,amsthm,mathrsfs}
\usepackage{epsfig,graphicx,subfigure}
\usepackage{enumerate,balance,mathtools}
\usepackage{wrapfig}
\usepackage{mathrsfs,euscript}
\usepackage[vlined,boxed,commentsnumbered,linesnumbered,ruled]{algorithm2e}
\usepackage{amsfonts}
\usepackage{amsthm}
\usepackage{fancyhdr}
\usepackage{picinpar}
\usepackage{listings}
\usepackage{layout}
\usepackage[all]{xy}
\usepackage{indentfirst}
\usepackage{tikz-cd}
\usepackage{tikz}
\usepackage[all]{xy}
\usepackage{mathrsfs}
\usepackage[colorlinks=true, allcolors=blue]{hyperref}
\usepackage{setspace}

\usepackage{amsmath,amsthm,amssymb,amsxtra,calligra,mathrsfs}
\usepackage{graphicx}
\usepackage{tikz}
\usepackage{tikz-cd} 
\usepackage{xcolor}
\usepackage{upgreek}
\usepackage{comment}
\usepackage{graphicx}
\usepackage{mathtools}
\usepackage{extarrows}
\usepackage{faktor}
\usepackage{mathrsfs,euscript}
\usepackage{comment}
\usepackage{bookmark}
\usepackage{hyperref}
\usepackage{mathrsfs}
\usepackage{multirow}

\hypersetup{
    colorlinks=true, 
    citecolor=blue,
    linkcolor=magenta,
    pdfauthor={},
	bookmarksnumbered=true,
}

\newcommand{\mb}[1]{\mathbb{#1}}

\newcommand{\ove}[1]{\overline{#1}}

\newcommand{\mtc}[1]{\mathcal{#1}}
\newcommand{\mtf}[1]{\mathfrak{#1}}
\newcommand{\mts}[1]{\mathscr{#1}}

\DeclareMathOperator{\Stab}{Stab}

\DeclareMathOperator{\Proj}{Proj}

\DeclareMathOperator{\mult}{mult}

\DeclareMathOperator{\Pic}{Pic}

\DeclareMathOperator{\rk}{rank}

\DeclareMathOperator{\SL}{SL}

\DeclareMathOperator{\che}{ch}

\DeclareMathOperator{\GIT}{GIT}

\DeclareMathOperator{\Aut}{Aut}

\DeclareMathOperator{\CM}{CM}

\newcommand{\sslash}{\mathbin{\mkern-3mu/\mkern-6mu/\mkern-3mu}}

\newcommand{\sheafHom}{\mathscr{H}\text{\kern -3pt {\calligra\large om}}\,}

\DeclareMathOperator{\topo}{top}

\newtheorem{theorem}{Theorem}[section]
\newtheorem{lemma}[theorem]{Lemma}

\newtheorem{corollary}[theorem]{Corollary}
\newtheorem{prop}[theorem]{Proposition}

\newtheorem*{notation}{Notation}

\newtheorem{remark}[theorem]{Remark}
\newtheorem{case}{Case}

\theoremstyle{remark}

\title{The Moduli Space of Genus Six Curves and K-stability: VGIT and the Hassett-Keel Program}
\date{}
\author{Junyan Zhao}
\address{851 S Morgan St, 60607, Chicago, Illinois, USA}
\email{jzhao81@uic.edu}

\begin{document}
\maketitle

\begin{abstract}
A general curve $C$ of genus six is canonically embedded into the smooth del Pezzo surface $\Sigma\subseteq \mb{P}^1\times\mb{P}^2$ of degree $5$ as a divisor in the class $\mtc{O}_{\Sigma}(2,2)$. In this article, we study the variation of geometric invariant theory (VGIT) for such pairs $(\Sigma,C)$, and relate the VGIT moduli spaces to the K-moduli of pairs $(\Sigma,C)$ and the Hassett–Keel program for moduli of genus six curves. We prove that the K-moduli spaces $\ove{M}^K(c)$ give the final several steps in the Hassett–Keel program for $\ove{M}_6$.

\end{abstract}

\tableofcontents

\section{Introduction}

This is the second article among three of the author in which we study the moduli space of curves of genus six using moduli spaces of pairs. A general genus six curve $C$ is canonically embedded into the smooth quintic del Pezzo surface $\Sigma$ as a divisor in the class $-2K_{\Sigma}$. Moreover, this embedding is unique up to $\Aut(\Sigma)$, which is isomorphic to the symmetric group $\mathfrak{S}_5$. Thus we expect a moduli space parameterizing pairs in which a general member is of the form $(\Sigma,C)$, where $C\in |-2K_{\Sigma}|$, to be related to the moduli space $\ove{M}_6$ of curves of genus six. A good candidate for such a moduli space is the K-moduli space $\ove{M}^K(c)$, $0<c<1/2$, parameterizing K-polystable pairs $(X,cD)$ which admit a $\mb{Q}$-Gorenstein smoothing to $(\Sigma,cC)$. This is the main object studied in \cite{zha22}.

For any $C\in|-2K_{\Sigma}|$, we may view $(\Sigma,C)$ as a complete intersection in $\mb{P}^1\times\mb{P}^2$ of type $(\mtc{O}_{\mb{P}^1\times\mb{P}^2}(1,2),\mtc{O}_{\mb{P}^1\times\mb{P}^2}(2,2))$. A natural parameter space for such complete intersections is a projective bundle $\mb{P}\mtc{E}\rightarrow \mb{P}^9$, which is equipped with a natural $G=\SL(2)\times \SL(3)$-action. The Picard rank of $\mb{P}\mtc{E}$ is 2, thus the GIT quotient involves a choice of
linearization parameterized by $t\in(0,1/2)$. We will denote these GIT quotients by $\ove{M}^{\GIT}(t)$. 

The case for curves of genus $4$ is analyzed in \cite{CMJL14}, in which the authors proved that the VGIT moduli spaces give the final steps in the Hassett-Keel program for genus 4 curves. Unfortunately, this cannot be expected to hold in our set-up. The main reason is the following: for any $X\in \mtc{O}_{\mb{P}^1\times\mb{P}^2}(1,2)$ such that $X\simeq \Sigma$, the subgroup of $\SL(2)\times \SL(3)$ fixing the point $[X]\in\mb{P}^9$ is an $\mtf{S}_4$, which is not isomorphic to $\Aut(\Sigma)=\mathfrak{S}_5$. The right thing we should expect is a generically $5:1$ map from the VGIT quotient to the moduli spaces which are birational to $\ove{M}_6$. Denote by $\mb{P}\mtc{E}^K(c)$ and $\mb{P}\mtc{E}^{\GIT}(t)$ the subsets of $\mb{P}\mtc{E}$ consisting of pairs $(X,D)$ which are c-K-semistable and GIT$_{t}$-semistable, respectively. Then we have the following result.

\begin{theorem} \textup{(Theorem \ref{12})}
    For any $c\in(0,1/17)$, set $t=t(c)=\frac{5c}{4+2c}$. Then we have $\mb{P}\mtc{E}^K(c)= \mb{P}\mtc{E}^{\GIT}(t)$, and identifications $$\ove{M}^{\GIT}(c)\simeq |-2K_{\Sigma}|\sslash \mtf{S}_4, \quad \textup{and} \quad \ove{M}^{K}(c)\simeq |-2K_{\Sigma}|\sslash \mtf{S}_5.$$ Moreover, we have a finite map $$\ove{M}^{\GIT}(c)\longrightarrow \ove{M}^{K}(c)$$ of degree $5$.
\end{theorem}

Moreover, by analyzing the GIT stability conditions carefully, we can find out all the GIT walls, i.e. the values $t_i$ such that the moduli spaces $\ove{M}^{\GIT}(t)$ change when $t$ varies in $(0,1/2)$ and crosses $t_i$. These walls turn out to be connected with walls for K-moduli spaces.

\begin{theorem} \textup{(Theorem \ref{13})}
    The VGIT walls are $$t_0=0,\quad t_1=\frac{1}{14},\quad t_2=\frac{1}{8},\quad t_3=\frac{1}{6},\quad t_4=\frac{1}{5},\quad t_5=\frac{1}{4},\quad t_6=\frac{1}{3},\quad t_7=\frac{1}{2},$$ which bijectively correspond via the relation $\displaystyle t(c)=\frac{5c}{4+2c}$ to the K-moduli walls $$c_0=0,\quad c_1=\frac{1}{17},\quad c_2=\frac{2}{19},\quad c_3=\frac{1}{7},\quad c_4=\frac{4}{23},\quad c_5=\frac{2}{9},\quad c_6=\frac{4}{13},\quad c_7=\frac{1}{2}.$$ Among the walls $t_i\in(0,1/2)$, only the one $t_1=\frac{1}{14}$ is a divisorial contraction, and the remaining 5 walls are flips.
\end{theorem}

Although the GIT quotients are not birational to $\ove{M}_6$, with the help of the projective bundle $\mb{P}\mtc{E}$, we can identify the K-moduli spaces with some of the log canonical models of $\ove{M}_6$.

Recall that for each $\alpha\in[0,1]$, we define a log canonical model of $\ove{M}_6$ to be $$\ove{M}_6(\alpha):=\Proj\bigoplus_{n\geq0}H^0\left(\ove{M}_6,n(K_{\ove{M}_6}+\alpha\delta)\right),$$ where $\delta=\delta_0+\delta_1+\delta_2+\delta_3$ is the boundary divisor $\ove{M}_6\setminus M_6$. For any genus $g\geq 4$, when varying $\alpha$ from $1$ to $0$, the first three walls are at $\alpha_1=\frac{9}{11}$, $\alpha_2=\frac{7}{10}$ and $\alpha_3=\frac{2}{3}$ (ref. \cite{HH09,HH13,AFS17}). For $g=6$, it is also known that the last non-trivial model is given by the GIT quotient $|-2K_{\Sigma}|\sslash \Aut(\Sigma)$ (ref. \cite{Mul14,Fed18}). The last main result in this paper is the following.

\begin{theorem}\label{16} \textup{(Theorem \ref{14})}
    Let $0\leq c\leq \frac{11}{52}$ be a rational number, and $\alpha(c)=\frac{32-19c}{94-68c}\in \left[\frac{16}{47},\frac{97}{276}\right]$. Then we have an isomorphism $$\ove{M}_6(\alpha(c))\simeq \ove{M}^K(c).$$ In particular, the last walls of the log canonical models $\ove{M}_6(\alpha)$ are $$\left\{\frac{16}{47},\frac{35}{102},\frac{29}{55},\frac{41}{118},\frac{22}{63},\frac{47}{134}\right\}.$$
\end{theorem}

\textbf{Acknowledgements} The author would like to thank Izzet Coskun, Maksym Fedorchuk, Yuchen Liu and Fei Si for many stimulating discussions, and Ben Gould for comments on the draft.

\section{Preliminaries}

In this section, we collect some facts about the projective bundle we consider and state some results in VGIT.

\subsection{Geometry of projective bundles}
Let $\mb{P}^{11}=\mb{P}H^0(\mb{P}^1\times \mb{P}^2,\mtc{O}_{\mb{P}^1\times \mb{P}^2}(1,2))$ be the projective space parameterizing quintic del Pezzo surfaces in $\mb{P}^1\times \mb{P}^2$, and $\mts{Q}\subseteq \mb{P}^1\times \mb{P}^2\times\mb{P}^{11}$ be the universal quintic. Let $\pi_1:\mb{P}^1\times \mb{P}^2\times\mb{P}^{11}\rightarrow\mb{P}^1\times \mb{P}^2$ and $\pi_2:\mb{P}^1\times \mb{P}^2\times\mb{P}^{11}\rightarrow\mb{P}^{11}$ be the two projections. Consider the locally free sheaf $$\mtc{E}:=\pi_{2*}(\mtc{O}_{\mts{Q}}\otimes\pi_{1}^{*}\mtc{O}_{\mb{P}^1\times\mb{P}^2}(2,2))$$ of rank 16 on $\mb{P}^{11}$ and set $\pi:\mb{P}\mtc{E}\rightarrow\mb{P}^{11}$ the corresponding projective bundle. Then $$\Pic(\mb{P}\mtc{E})\simeq \mb{Z}\eta\oplus \mb{Z}\xi,$$ where $\eta=\pi^{*}\mtc{O}_{\mb{P}^{11}}(1)$ and $\xi=\mtc{O}_{\mb{P}\mtc{E}}(1)$.

We first compute the Chern characters of $\mtc{E}$. Notice that $\mtc{E}$ fits into the short exact sequence $$0\longrightarrow \pi_{2*}(\mtc{I}_{\mts{Q}}\otimes\pi_{1}^{*}\mtc{O}_{\mb{P}^1\times \mb{P}^2}(2,2))\longrightarrow \pi_{2*}(\pi_{1}^{*}\mtc{O}_{\mb{P}^1\times \mb{P}^2}(2,2))\longrightarrow \mtc{E}\longrightarrow 0,$$ and that $$\pi_{2*}(\mtc{I}_{\mts{Q}}\otimes\pi_{1}^{*}\mtc{O}_{\mb{P}^1\times \mb{P}^2}(2,2))=H^0(\mb{P}^1\times \mb{P}^2,\mtc{O}_{\mb{P}^1\times \mb{P}^2}(1,0))\otimes \mtc{O}_{\mb{P}^{11}}(-1),$$ $$\pi_{2*}(\pi_{1}^{*}\mtc{O}_{\mb{P}^1\times \mb{P}^2}(2,2))=H^0(\mb{P}^1\times \mb{P}^2,\mtc{O}_{\mb{P}^1\times \mb{P}^2}(2,2))\otimes \mtc{O}_{\mb{P}^{11}}.$$  It follows that $$\che(\mtc{E})=18-2\che(\mtc{O}_{\mb{P}^9}(-1))=18-2\sum_{i=0}^{\infty}\frac{(-1)^i\eta^i}{i!}.$$ In particular, we have $c_1(\mtc{E})=2\eta$ and hence the following result. 

\begin{prop}
    The canonical line bundle of $\mb{P}\mtc{E}$ is $$K_{\mb{P}\mtc{E}}\sim -14\eta-16 \xi.$$ 
\end{prop}

\begin{proof} 

By \cite[Section 7.3.A]{Laz04}, we have that $$K_{\mb{P}\mtc{E}}=\pi^{*}(\omega_{\mb{P}^{11}}\otimes \det(\mtc{E}^{*}))\otimes \mtc{O}_{\mb{P}\mtc{E}}(-\rk\mtc{E})=-14\eta-16\xi.$$
\end{proof}

Let $U\subseteq \mb{P}\mtc{E}$ be the Zariski big open subset parameterizing complete intersections $(X,D)$ with $X$ irreducible. Let $i:(\mathscr{X},\mts{D})\hookrightarrow \mb{P}^1\times\mb{P}^2\times U$ be the tautological family of the pairs, $p_1,p_2$ the projection maps from $(\mb{P}^1\times\mb{P}^2)\times U$ to the two factors, respectively, and $f:=p_2\circ i$. In summary, we have the following commutative diagram
    $$ \xymatrix{
    (\mts{X},\mts{D}) \ar[dd]_{f} \ar@{^(->}[r]^{i\quad} & (\mb{P}^1\times\mb{P}^2)\times U \ar[ddl]^{p_2} \ar[r]^{\quad p_1} &  \mb{P}^1\times\mb{P}^2 \\
    \\
    U  \ar[r]^{\pi} & \mb{P}^{11} }.$$

\begin{prop}
The nef cone of $\mb{P}\mtc{E}$ is generated by $\eta$ and $\eta+\frac{1}{2}\xi$.
\end{prop}

\begin{proof}
    It is clear that $\eta$ is an extremal ray. The argument in \cite[Theorem 2.7]{Ben14} shows that the other extremal ray is $\eta+\frac{1}{2}\xi$.
\end{proof}

\subsection{Variation of GIT}
Now our main object is a polarized variety $(X,L)$ together with an action by a reductive group $G$.

\begin{lemma}\textup{(ref. \cite[Lemma 3.10]{Laz13})}\label{8}
Let $(X,L)$ be a polarized projective variety, and $G$ a reductive group acting on $(X,L)$. Let $L_0$ be a $G$-linearized line bundle on $X$. For any rational number $0<\varepsilon\ll1$, we define $L_{\pm \varepsilon}:=L\otimes L_0^{\otimes (\pm \varepsilon)}$. Let $X^{ss}(0)$ and $X^{ss}(\pm)$ denote the semistable loci, respectively. Then
\begin{enumerate}[(i)]
\item there are open immersions $X^{ss}(\pm)\subseteq X^{ss}(0)$, and
\item for any $x\in X^{ss}(0)\setminus X^{ss}(\pm)$, there is a 1-parameter subgroup $\sigma$ of $G$ such that $$\mu^{L}(x,\sigma)=0,\quad \textup{and}\quad \mu^{L_{\pm}}(x,\sigma)<0.$$
\end{enumerate}
\end{lemma}

For simplicity, we denote $X\sslash_{L_{\bullet}}G$ by $X\sslash_{\bullet}G$, where $\bullet\in\{+,-,0\}$. As a consequence, the functorial properties of the GIT quotients give the following commutative diagram:

  $$ \xymatrix{
    X^{s}(L_{\pm})/G \ar@{^(->}[dd]  & &X^{s}(L_0)/G \ar@{_(->}[ll] \ar@{^(->}[dd]  \\
    \\ 
    X\sslash_{\pm}G  \ar[rr]^{\varphi_{\pm}} & & X\sslash_{0}G },$$ where $X^{s}(L_{\bullet})$ are the stable loci, and $X^s(L_{\bullet})/G$ are geometric quotients.

\begin{theorem} \textup{(ref. \cite[Theorem 3.3]{Tha96})} With notation as in Lemma \ref{8}, if both $X\sslash_{-}G$ and $X\sslash_{+}G$ are non-empty, then $\varphi_{-}$ and $\varphi_{+}$ are proper and birational. If they are both small contractions, then the rational map $$g:X\sslash_{-}G\dashrightarrow X\sslash_{+}G$$ is a flip with respect to $\mtc{O}(1)$ on $X\sslash_{+}G$, which is the relative bundle of the projective morphism $g:X\sslash_{+}G\rightarrow X\sslash_{0}G$.
\end{theorem}

\section{Stability of points in $\mb{P}\mtc{E}$}

In this section, we determine the stability conditions on the projective bundle $\mb{P}\mtc{E}$ using the Hilbert-Mumford (abbv. HM) criterion as the slope $t$ varies. 

For any point $(f,[g])\in \mb{P}\mtc{E}$, we will not distinguish the pair $(f,[g])$ of equations and the corresponding pair $(X,D)$ of schemes. For any polarization $L=a\eta+b\xi$ and any 1-parameter subgroup (abbv. 1-PS) $\lambda$, we denote by $$\mu^{\frac{b}{a}}(X,D;\lambda)=\mu^{a\eta+b\xi}(X,D;\lambda)$$ the HM-invariant of the point $(X,D)$. The following result about the HM-invariant will be widely used.

\begin{prop} \textup{(ref. \cite[Proposition 2.15]{Ben14})} The HM-invariant of a point $(f,[g])=\left(X,D\right)\in\mb{P}\mtc{E}$ is given by $$\mu^{a\eta+b\xi}((X,D);\lambda)=a\mu(f;\lambda)+b\mu(g;\lambda),$$ where $g\in H^0(\mb{P}^1\times \mb{P}^2,\mtc{O}_{\mb{P}^1\times \mb{P}^2}(2,2))$ is a representative of $[g]$ with minimal $\lambda$-weight.
    
\end{prop}

\subsection{Stability thresholds for singular surfaces}

The main goal in this part is to show that for each type of singular surface $X$ of  class $\mtc{O}_{\mb{P}^1\times\mb{P}^2}(1,2)$, the pair $(X,D)$ is not t-semistable for any $D\in|-2K_X|$ when $t$ is small. We will see later that the bounds given in this section are indeed the stability thresholds.

\begin{lemma}\label{11}
    Let $X\subseteq \mb{P}^1_{u,v}\times \mb{P}^2_{x,y,z}$ be an irreducible surface of  class $\mtc{O}_{\mb{P}^1\times \mb{P}^2}(1,2)$, then $X$ satisfies one of the followings:\textup{
    \begin{enumerate}[(i)]
        \item $X$ is smooth;
        \item $X$ has a unique $A_1$-singularity;
        \item $X$ has exactly two $A_1$-singularities;
        \item $X$ has a unique $A_2$-singularity.
        \item $X$ has a unique $A_3$-singularity.
        \item $X$ is singular along a line, which is a fiber of $\mb{P}^1\times\mb{P}^2\rightarrow\mb{P}^2$.
    \end{enumerate}}
\end{lemma}

\begin{proof}
    We may assume that $$X=\{uf_2(x,y,z)+vg_2(x,y,z)=0\}$$ has a singularity at $p=((0:1),(1:0:0))$. Then $f_2$ has no $x^2$ term, and $g_2$ has no $x^2,xy,xz$. After a change of coordinates via a linear action on $(y,z)$, we may also assume that $f_2$ has no $xy$ term. Notice that the coefficient of $xz$ and $y^2$ in $f_2$ is non-zero: otherwise $X$ will either be singular along the line $((u:v),(1:0:0))$ or have an irreducible component $\{z=0\}$, a contradiction to the irreducibility of $X$. Thus if $g_2$ has a non-zero term $y^2$, then $X$ has an $A_1$-singularity at $p$. In this case, we may assume $X$ has another singularity $q=((u_q:v_q),(x_q:y_q:z_q))$. Taking partial derivative for $x$, we see that either $u_q=0$ or $z_q=0$. 
    \begin{enumerate}[(a)]
        \item If $u_q=0$, then $(y_q:z_q)$ has to be a singular point of $g_2$. If $g_2$ is reduced, then $z_q=y_q=0$, and hence $q=p$, which is a contradiction. Thus $g_2=(ay+bz)^2$, where $a\neq 0$. In this case, $X$ has two $A_1$ singularities at $p$ and $q$. 
        \item If $z_q=0$, then taking the $v$ derivative, one gets that either $y_q=0$ because $g_2$ has a non-zero $y^2$ term. In the former case, taking the $x$ derivative, one obtains that $u_q=0$ and hence $q=p$. 
    \end{enumerate}

    Now we assume that $g_2$ has no $y^2$ term, and thus $g_2=z(ay+bz)$. There are also two cases:
    \begin{enumerate}[(a)]
        \item $a=0$, and hence $\{g_2=0\}$ is a double line tangent to the smooth conic $\{f_2=0\}$ at the point $p$. In this case, $p$ is the only singularity of $X$, which is of $A_3$-type.
        \item $a\neq0$, and thus $\{g_2=0\}=L_1\cup L_2$ is a union of two lines in $\mb{P}^2$, where $L_1$ is tangent to the smooth conic $\{f_2=0\}$ at the point $p$, and $L_2$ passes through $p$ and intersects the conic properly. In this case, $p$ is the only singularity of $X$, which is of $A_2$-type.
    \end{enumerate}
\end{proof}

For any rational number $0<t<1/2$, we say that $(X,D)$ is \emph{t-(semi)stable} if $(X,D)$ is GIT-(semi)stable with respect to the polarization $L_t:=\eta+t\xi$. By the HM criterion, this is equivalent to saying that $\mu^t(X,D;\lambda)\geq0$ for any 1-PS $\lambda$ of $(X,D)$. We say that a 1-PS $\lambda$ induced by the $\mb{G}_m$ action of weight $(t_0,t_1;s_0,s_1,s_2)$, or for short a 1-PS $\lambda$ of weight $(t_0,t_1;s_0,s_1,s_2)$, if $$\lambda\cdot((u:v),(x:y:z))=((\lambda^{t_0}u:\lambda^{t_1}v),(\lambda^{s_0}x:\lambda^{s_1}y:\lambda^{s_2}z)).$$

We first show that if $X$ has non-isolated singularities, then $(X,D)$ is never t-semistable for $0<t<1/2$ and $D\in |-2K_X|$.

\begin{prop}\label{6}
   If a point $([f],[g])\in\mb{P}\mtc{E}$ is t-semistable for some $0<t<\frac{1}{2}$, then $X=\mb{V}(f)$ is irreducible. 
\end{prop}

\begin{proof}
    Notice that if $X$ is reducible, then it has a component of the class either $\mtc{O}_{\mb{P}^1\times \mb{P}^2}(0,1)$ or $\mtc{O}_{\mb{P}^1\times \mb{P}^2}(1,0)$. In the former case, we may assume this component is defined by $\{x=0\}$. Consider the action $\lambda$ of weight $(0,0;-2,1,1)$. Then $\mu(f;\lambda)\leq -1$ and $\mu(g;\lambda)\leq 2$ so that $\mu^t(X,D;\lambda)<0$ for any $t<\frac{1}{2}$. In the latter case, assuming this component is defined by $\{u=0\}$, the action $\lambda$ of weight $(-1,1;0,0,0)$ satisfies $\mu(f;\lambda)=-1$ and $\mu(g;\lambda)\leq 2$ so that $\mu^t(X,D;\lambda)<0$ for any $t<\frac{1}{2}$.
\end{proof}

\begin{prop}\label{7}
    If $X\subseteq \mb{P}^1\times \mb{P}^2$ is singular along a line, which is a fiber over $\mb{P}^2$, then $X$ is t-unstable for any $0<t<1/2$.
\end{prop}

\begin{proof}
    The defining equation of $X$ is of the form $uf_2(y,z)+vf_2(y,z)=0$. Consider the 1-PS $\lambda$ induced by the $\mb{G}_m$-action of weight $(0,0;2,-1,-1)$. Then $\mu(X;\lambda)=-2$, and hence $$\mu^t(X,D;\lambda)\leq -2+4t<0$$ for any $0<t<1/2$.
\end{proof}

\begin{corollary}\label{15}
For any $0<t<\frac{1}{2}$, the locus of t-semistable points $(X,D)\in \mb{P}\mtc{E}$ is contained in the open subset $U$ of $\mb{P}\mtc{E}$ parameterizing complete intersections. Moreover, for any t-semistable point $(X,D)$, the surface has only isolated singularities of type $A_1$, $A_2$ or $A_3$. 
\end{corollary}

\begin{prop}\label{2}
    Suppose $X$ is an irreducible singular surface with an $A_1$-singularity. Then $(X,D)$ is t-unstable for any $D\in H^0(\mb{P}^1\times \mb{P}^2,\mtc{O}_{X}(2,2))$ and $0<t<\frac{1}{14}$.
\end{prop}

\begin{proof}
    We may assume $X$ is given by the equation $uyz+vx(x+y+z)=0$ with an $A_1$-singularities $p=((0:0:1),(0:1))$. Consider the action $\lambda$ of weight $(-3,3;-2,4,-2)$. Then we have $\mu(X,\lambda)=-1$, and $$\mu^t(X,D;\lambda)\leq -1+14t<0$$ for $0<t<\frac{1}{14}$.
\end{proof}

\begin{prop}
    Let $X\subseteq \mb{P}^1\times\mb{P}^2$ be a quintic del Pezzo surface with two $A_1$-singularities at $p$ and $q$, and $D\in |-2K_X|$ a boundary divisor. Then for $t<\frac{1}{8}$, the point $(X,D)$ is t-unstable.
\end{prop}

\begin{proof}
    We may assume $X$ is given by the equation $uyz+vx^2=0$ with two $A_1$-singularities at $p=((0:0:1),(0:1))$ and $q=((0:1:0),(0:1))$. Consider the action $\lambda$ of weight $(-3,3;-2,1,1)$. Then we have $\mu(f,\lambda)=-1$, and $$\mu^t(X,D;\lambda)\leq -1+8t<0$$ for $0<t<\frac{1}{8}$.
\end{proof}

\begin{prop}
    Let $X\subseteq \mb{P}^1\times\mb{P}^2$ be a quintic del Pezzo surface with an $A_2$-singularity at $p$, and $D\in |-2K_X|$ a boundary divisor. Then for $t<\frac{1}{6}$, the point $(X,D)$ is t-unstable.
\end{prop}

\begin{proof}
    We may assume $X$ is given by the equation $uyz+v(xz-y^2)=0$ with an $A_2$-singularity at $p=((0:0:1),(1:0))$. Consider the action $\lambda$ of weight $(1,-1;-2,0,2)$. Then we have $\mu(f,\lambda)=-1$, and $$\mu^t(X,D;\lambda)\leq -1+6t<0$$ for $0<t<\frac{1}{6}$.
\end{proof}

\begin{prop}
    Let $X\subseteq \mb{P}^1\times\mb{P}^2$ be a quintic del Pezzo surface with an $A_3$-singularity at $p$, and $D\in |-2K_X|$ a boundary divisor. Then for $t<\frac{1}{4}$, the point $(X,D)$ is t-unstable.
\end{prop}

\begin{proof}
    We may assume $X$ is given by the equation $ux^2+v(xz-y^2)=0$ with an $A_3$-singularity at $p=((0:0:1),(1:0))$. Consider the action $\lambda$ of weight $(1,-1;-1,0,1)$. Then we have $\mu(f,\lambda)=-1$, and $$\mu^t(X,D;\lambda)\leq -1+4t<0$$ for $0<t<\frac{1}{4}$.
\end{proof}

\subsection{Surfaces with exactly one singularity}

From now on, we take the boundary divisor $D$ into consideration. We will determine the value $T$ for which a pair $(X,D)$ is $T$-semistable but not t-semistable for any $t<T$. Since the surface $X$ with two $A_1$-singularities is a toric surface, which is more complicated than the others, we first deal with the surfaces with a single singularity.

\begin{lemma}\label{5}
    Let $X\subseteq \mb{P}^1\times\mb{P}^2$ be the quintic del Pezzo surface defined by the equation $$uyz+vx(y+z)=0,$$ which has an $A_1$-singularity at $p=((1:0:0),(1:0))$. Let $D\in |-2K_X|$ be the curve on $X$ which contains a line of multiplicity $4$, defined by the equation $u^2x^2=0$. Then $(X,D)$ is t-semistable if and only if $t=1/14$.   
\end{lemma}

\begin{proof}
    Consider the two one-parameter subgroups given by the $\mb{G}_m$-actions of weights $(3,-3;4,-2,-2)$ and $(-3,3;-4,2,2)$. We see that if $(X,D)$ is t-semistable, then  $t=1/14$. Now we prove the converse: the HM-invariant is non-negative with respect to any 1-PS. Let $\lambda$ be a 1-PS, and we will discuss by cases.
    \begin{enumerate}[(1)]
        \item Suppose the weight of $\lambda$ is $(0,0;a+b,-a,-b)$. Then we have $$\mu(X;\lambda)=\max\{a,b,-a-b\}.$$ We may moreover assume $a\geq b$. There are three subcases.
        \begin{enumerate}[(i)]
            \item $a\geq b\geq -a-b$: we have $\mu(X;\lambda)=a$ and $\mu(D;\lambda)=2(a+b)\geq a$. It follows that $$\mu^t(X,D;\lambda)\geq \frac{15}{14}a\geq0$$ for $t=1/14$.
            \item $-a-b\geq a\geq b$: we have $\mu(X;\lambda)=-a-b$ and $\mu(D;\lambda)=2(a+b)$. It follows that $$\mu^t(X,D;\lambda)\geq-\frac{6}{7}(a+b)\geq0$$for $t=1/14$.
            \item $a\geq -a-b\geq b$: we have $\mu(X;\lambda)=a$ and $\mu(D;\lambda)=a\geq 2(a+b)\leq a$. It follows that $$\mu^t(X,D;\lambda)=a+\frac{1}{7}(a+b)=\frac{6}{7}a+\frac{1}{7}(2a+b)\geq 0$$ for $t=1/14$.
        \end{enumerate}
     \item Suppose the weight of $\lambda$ is $(1,-1;a+b,-a,-b)$. Then we have $$\mu(X;\lambda)=\max\{a-1,b-1,1-a-b\}.$$ Again, we can assume that $a\geq b$. There are still three subcases to discuss.   
     \begin{enumerate}[(i)]
            \item $a\geq b\geq -a-b$ : If $a-1\geq 1-a-b$, i.e. $b\geq 2-2a$, then we have $\mu(X;\lambda)=a-1$ and $\mu(D;\lambda)=2+2(a+b)$. It follows that $$\mu^\frac{1}{14}(X,D;\lambda)=a-1+\frac{1}{7}(1+a+b)\geq a-1+\frac{1}{7}(3-a)=\frac{2}{7}(3a-2) \geq0$$ since we have $a\geq b\geq 2-2a$. Similarly, if $1-a-b\geq a-1$, i.e. $b\leq 2-2a$, then we have $\mu(X;\lambda)=1-a-b$ and $\mu(D;\lambda)=2+2(a+b)$. It follows that 
            \begin{equation}\nonumber
            \begin{split}
                \mu^\frac{1}{14}(X,D;\lambda)&=1-a-b+\frac{1}{7}(1+a+b)\\
                &= \frac{2}{7}(4-3a-3b)\\
                &=\frac{2}{7}((4-4a-2b)+(a-b))\geq0.
            \end{split}    
          \end{equation} 
            \item $-a-b\geq a\geq b$ : we have $\mu(X;\lambda)=1-a-b$ and $\mu(D;\lambda)=2+2(a+b)$. It follows that 
            \begin{equation}\nonumber
            \begin{split}
                \mu^\frac{1}{14}(X,D;\lambda)&=1-a-b+\frac{1}{7}(1+a+b)\\
                &= \frac{2}{7}(4-3a-3b)\geq0,
            \end{split}    
          \end{equation} since in this case we always have $a+b\geq 0$. 
            \item $a\geq -a-b\geq b$ : If $a-1\geq 1-a-b$, i.e. $b\geq 2-2a$, then we have $\mu(X;\lambda)=a-1$ and $\mu(D;\lambda)=2+2(a+b)$. It follows that $$\mu^\frac{1}{14}(X,D;\lambda)=a-1+\frac{1}{7}(1+a+b)\geq a-1+\frac{1}{7}(3-a)=\frac{2}{7}(3a-2) \geq0$$ since we have $a\geq b\geq 2-2a$. Similarly, if $1-a-b\geq a-1$, i.e. $b\leq 2-2a$, then we have $\mu(X;\lambda)=1-a-b$ and $\mu(D;\lambda)=2+2(a+b)$. It follows that 
            \begin{equation}\nonumber
            \begin{split}
                \mu^\frac{1}{14}(X,D;\lambda)&=1-a-b+\frac{1}{7}(1+a+b)\\
                &= \frac{2}{7}(4-3a-3b)\\
                &=\frac{2}{7}((4-4a-2b)+(a-b))\geq0.
            \end{split}    
          \end{equation} 
        \end{enumerate}
     \item The case when the weight of $\lambda$ is $(-1,1;a+b;-a;-b)$ is similar to the case (2), so we omit it here.   
    \end{enumerate}
\end{proof}

\begin{prop}\label{9}
    Let $X\subseteq \mb{P}^1\times\mb{P}^2$ be a quintic del Pezzo surface with exactly one $A_1$-singularity at $p$, and $D\in |-2K_X|$ a boundary divisor. Then for $t=\frac{1}{14}+\varepsilon$, the point $(X,D)$ is t-semistable if and only if $p\notin D$ and $D$ contains no curves of multiplicity $4$.
\end{prop}

\begin{proof}
    We may assume $X$ is given by the equation $uyz+vx(y+z)=0$ with an $A_1$-singularity at $p=((1:0:0),(1:0))$. If $D$ passes through $p$, then any defining equation of $D$ does not contain the term $u^2x^2$. Taking the 1-PS $\lambda$ of weight $(3,-3;4,-2,-2)$, one has that $\mu(X,\lambda)=-1$ and $\mu(D;f)\leq 8$ so that $\mu^t(X,D;\lambda)<0$. In fact, the same argument shows that if $\mult_pD\geq 2$, then $(X,D)$ is t-unstable for any $0<t<1/2$.
    
    Now we may assume $D$ does not contain $p$ and $D\neq D_0:=\{x^2u^2=0\}$ by Lemma \ref{5}. Then any defining equation for $D$ has a non-zero $x^2u^2$ term. The $G_m$-action $$\lambda((u:v),(x:y:z))=((\lambda u:\lambda^{-1}v),(\lambda^2x:y:z))$$ induces a 1-PS, where the central fibre is $(X,D_0)$ and any other fibre is isomorphic to $(X,D)$. By Lemma \ref{5} and openness of GIT-semistability, we have that $(X,D)$ is t-semistable for $t=1/14$. The wall-crossing structure for VGIT implies that at least one (and hence a general one) of the $D\neq D_0$ such that $p\notin D$ satisfies that $(X,D)$ is t-semistable for $t=1/14+\varepsilon$. Moreover, since for a general such $D$, the automorphism group of $(X,D)$ is finite, then $(X,D)$ is in fact t-stable for $t=1/14+\varepsilon$. Notice that we have an exact sequence of group scheme $$0\longrightarrow \Aut^0(X)=\mb{G}_m\longrightarrow \Stab(X)\longrightarrow G_0\longrightarrow0,$$ where $\Stab(X)$ is the stabilizer of $X$ under the $\SL(2)\times \SL(3)$-action and $G_0$ is a finite group. Let $V$ be the affine subspace of $H^0(X,\mtc{O}_{X_n}(2,2))$ consisting of polynomials whose coefficient of $x^2u^2$ is $1$. Viewing $V$ as a vector space, the $G_m$-action on $V\setminus \{0\}$ is induced by $$t((u:v),(x:y:z))=((tu:t^{-1}v),(t^2x:y:z)).$$ Then $(V-\{0\})/\mb{G}_m$ is a weighted projective space $\mb{P}(1^2,2^5,3^5,4^3)$ with a $G$-action. Let $E_{+}$ be the exceptional locus of the wall-crossing morphism $$\ove{M}^{\GIT}(1/14+\varepsilon)\rightarrow \ove{M}^{\GIT}(1/14).$$ Observe that there is an injective dominant morphism $E_{+}\hookrightarrow \mb{P}(1^2,2^5,3^5,4^3)/G$. Since both of them are proper and normal, and the later is irreducible, then this is an isomorphism by Zariski's Main Theorem (ref. \cite[Theorem 3.11.4]{Har13}).
    
\end{proof}

It is seen in the proof to Proposition \ref{9} that for the quintic del Pezzo surface $X$ with exactly one $A_1$-singularity, and $D\in |-2K_X|$ such that $p\in D$, we have that $(X,D)$ is t-unstable for any $t<1/2$ if $\mult_pD\geq 2$; for any $t<1/8$ if $\mult_pD=1$. Now we prove the following result, which is a complement to this statement.

\begin{prop}\label{10}
    Let $X\subseteq \mb{P}^1\times\mb{P}^2$ be a quintic del Pezzo surface $X$ with exactly one $A_1$-singularity, and $D\in |-2K_X|$ such that $\mult_pD=1$. Then $(X,D)$ is t-semistable for $t=1/8$.
\end{prop}

\begin{proof}
    The condition $\mult_pD=1$ comes down to saying that for any defining equation $g$ of $D$, $g$ has no $x^2u^2$ term, but at least one of the coefficients of $u^2xy,u^2xz,uvx^2$ is non-zero. Notice that for any $D$ such that $\mult_pD=1$, there is a $\mb{G}_m$-action inducing a degeneration from $(X,D)$ to $(X,D_0)$, where $D_0$ has only $u^2xy,u^2xz,uvx^2$. Thus by the openness of GIT-semistability, it suffices to show that $(X,D_0)$ is $\frac{1}{8}$-semistable for any $D_0$ given by the equation $$au^2xy+bu^2xz+cuvx^2=0,$$ where $(a:b:c)\in \mb{P}^2$. Indeed, this is true using the same argument as in the proof of Lemma \ref{5}.
\end{proof}

The next result involves the surfaces with an $A_2$ or an $A_3$ singularity. The proof is identical to that of Proposition \ref{10}, hence we wrap up the statements together and omit the details of the proof but only give the 1-parameter subgroups that imposes conditions on the stability thresholds.

\begin{prop}
    Let $X\subseteq \mb{P}^1\times\mb{P}^2$ be a quintic del Pezzo surface with exactly one singularity $p$, and $D\in |-2K_X|$ a boundary divisor.
    \begin{enumerate}[(1)]
        \item Suppose $X=\{uyz+v(xz-y^2)=0\}$ and $p$ is an $A_2$-singularity. 
        \begin{enumerate}[(i)]
            \item If $p\notin D$, then $(X,D)$ is $\frac{1}{6}$-semistable but not $t$-semistable for $t<1/4$;
            \item if $\mult_pD=1$, then $(X,D)$ is $\frac{1}{4}$-semistable but not $t$-semistable for $t<1/4$; 
            \item if $\mult_pD\geq 2$, then $(X,D)$ is t-unstable for any $0<t<1/2$.
        \end{enumerate}  
        \item Suppose $X=\{ux^2+v(xz-y^2)=0\}$ and $p$ is an $A_3$-singularity. 
        \begin{enumerate}[(i)]
            \item If $p\notin D$, then $(X,D)$ is $\frac{1}{4}$-semistable but not $t$-semistable for $t<1/4$;
            \item if $\mult_pD=1$, then $(X,D)$ is $\frac{1}{3}$-semistable but not $t$-semistable for $t<1/3$; 
            \item if $\mult_pD\geq 2$, then $(X,D)$ is t-unstable for any $0<t<1/2$.
        \end{enumerate} 
    \end{enumerate}
    
    Then for $1/4<t<1/3$, the point $(X,D)$ is t-semistable if and only if $p\notin D$ and $D\neq $
\end{prop}

\begin{proof}
    Consider the $\mb{G}_m$-action of weight $(1,-1:-2,0,2)$ for (1), and weight $(1,-1;-1,0,1)$ for (2).
\end{proof}

\subsection{Surfaces with two singularities}

For the rest of the this section, let us deal with the surface $X$ with exactly two $A_1$-singularities $p,q$. The complexity stems from the fact that this surface is toric, and hence has a number of $\mb{G}_m$-actions.

As before, we may assume the defining equation of $X$ is $uyz+vx^2=0$, and hence $p=((0:0:1),(0:1))$ and $q=((0:1:0),(0:1))$.

\begin{lemma}
  Let $D\in |-2K_X|$ be a divisor. If either $\mult_pD\geq 2$ or $\mult_qD\geq 2$, then $(X,D)$ is t-unstable for any $0<t<1/2$.
\end{lemma}

\begin{proof}
    We may assume that $\mult_pD\geq 2$, i.e. for any defining equation $g$ of $D$, the coefficients of $v^2z^2,v^2xz,v^2yz$ and $uvz^2$ in $g$ are all zero. Consider the 1-PS $\lambda$ of weight $(-3,3;-2,-2,4)$. We have that $\mu(X;\lambda)=-1$ and $\mu(D;\lambda)\leq 2$, thus $\mu^t(X,D;\lambda)<0$ for any $0<t<1/2$.
\end{proof}

Observe that $p\notin D$ comes down to saying that any defining equation $g$ of $D$ has non-zero $z^2v^2$ term, and $\mult_pD\leq 1$ is equivalent to saying that at least one of the coefficients of $z^2v^2,v^2xz,v^2yz,z^2uv$ is non-zero. A similar equivalent condition for $q\notin D$ also holds.

\begin{prop}
  Let $D\in |-2K_X|$ be a divisor. Then $(X,D)$ is $\frac{1}{8}$-semistable if either $p,q\notin D$ or a defining equation of $D$ has a non-zero $v^2yz$ term. 
\end{prop}

\begin{proof}
    Notice that any such pair specializes to the pair $(X,D_0)$ (i.e. one can find a 1-PS, where general fibers are isomorphic to $(X,D)$, and the central fiber is $(X,D_0)$), where $D_0$ is defined by the equation $v^2f(y,z)$ for some quadratic homogeneous function in $y$ and $z$ which is not equal to $y^2$ or $z^2$. By the openness of GIT-semistability, it suffices prove for the pair $(X,D_0)$. This is true by the same argument as in the proof of Lemma \ref{5}.
\end{proof}

To make our notation succinct, we set $a_{x^2uv}$ be the coefficient of $x^2uv$ in a defining equation of $D$, and same for the other monomial terms. Now we only need to consider the divisor $D\in |-2K_X|$ such that either $a_{v^2yz}=a_{v^2z^2}=0$ or $a_{v^2yz}=a_{v^2y^2}=0$. We may assume the latter case. 

\begin{case}
\textup{$a_{y^2v^2}\neq0$ and $a_{v^2xz}\neq0$:}
\end{case}

Each such pair specializes via the $\mb{G}_m$-action of weight $(-3,3;-2,0,2)$ to $(X,D_1)$ where $D_0$ is defined by an equation $v^2(y^2-axz)$ for some $a\neq 0$. Consider the same $\mb{G}_m$-action on the pair $(X,D_1)$, we see that it is t-unstable for any $t<1/6$. In fact, using the same argument as in Lemma \ref{5}, we can prove the following.
\begin{lemma}
    The pair $(X,D_1)$ is $\frac{1}{6}$-semistable but not t-semistable for any $t<\frac{1}{6}$.
\end{lemma}

\begin{case}
\textup{$a_{y^2v^2}\neq0$, $a_{v^2xz}=0$ and $a_{uvz^2}\neq0$:}
\end{case}

Each such pair specializes via the $\mb{G}_m$-action of weight $(-6,6;-4,-1,5)$ to $(X,D_2)$ where $D_2$ is defined by an equation $v(vy^2+az^2u)$ for some $a\neq 0$. As in the previous case, we have the following result. 
\begin{lemma}
    The pair $(X,D_3)$ is $\frac{1}{5}$-semistable but not t-semistable for any $t<\frac{1}{5}$.
\end{lemma}

For the remaining cases, we can assume moreover $a_{y^2v^2}=0$.

\begin{case}
\textup{$a_{v^2xy}\neq0$ and $a_{v^2xz}\neq0$:}
\end{case}

Each such pair specializes via the $\mb{G}_m$-action of weight $(-3,3;-2,1,1)$ to $(X,D_3)$ where $D_3$ is defined by an equation $v^2x(y+az)$ for some $a\neq 0$. Similar as above case, we have the following result.
\begin{lemma}
    The pair $(X,D_3)$ is $\frac{1}{5}$-semistable but not t-semistable for any $t<\frac{1}{5}$.
\end{lemma}

Now we can assume at least one of $a_{v^2xy}$ and $a_{v^2xz}$ is zero. Suppose $a_{v^2xz}=0$. There are two remaining cases.

\begin{case}
\textup{$a_{z^2uv}\neq0$ and $a_{v^2xy}\neq0$:}
\end{case}

Each such pair specializes via the $\mb{G}_m$-action of weight $(-3,3;-2,0,2)$ to $(X,D_4)$ where $D_4$ is defined by an equation $v(z^u+avxy)$ for some $a\neq 0$. Similarly, we have the following result.
\begin{lemma}
    The pair $(X,D_4)$ is $\frac{1}{4}$-semistable but not t-semistable for any $t<\frac{1}{4}$.
\end{lemma}

\begin{case}
\textup{$a_{z^2uv}\neq0$, $a_{y^2uv}\neq0$ and $a_{v^2xy}=0$:}
\end{case}

In this case, we have $\mult_pD,\mult_qD\geq 2$, and thus $(X,D)$ is t-unstable for any $t<1/2$.

\section{Relation to K-moduli spaces}

\subsection{Computation of CM-line bundles}

The main purpose of this section is to compute the CM $\mb{Q}$-line bundle for the family $(\mts{X},c\mts{D})\rightarrow U\subseteq \mb{P}\mtc{E}$, where $U$ is the Zariski big open subset of $\mb{P}\mtc{E}$ parameterizing complete intersections $(X,D)$ with $X$ irreducible. Moreover, we will show that the CM $\mb{Q}$-line bundle is proportional to the polarization line bundle for taking VGIT. 

\begin{prop}
    We have that $$-f_{*}(-K_{\mts{X}/U}-c\mts{D})^3=3(1-2c)^2(4+2c)\left(\eta+\frac{5c}{4+2c}\xi\right).$$
\end{prop}

\begin{proof}
    It follows from our construction that $$\mtc{O}_{\mb{P}^1\times\mb{P}^2\times U}(\mts{X})=p_1^{*}\mtc{O}_{\mb{P}^1\times\mb{P}^2}(1,2)\otimes p_2^{*}\pi^{*}\mtc{O}_{\mb{P}^{11}}(1),$$ $$\mtc{O}_{\mts{X}}(\mts{D})=p_1^{*}\mtc{O}_{\mb{P}^1\times\mb{P}^2}(2,2)|_{\mts{X}}\otimes p_1^{*}\mtc{O}_{U}(1)|_{\mts{X}}.$$ By the adjunction formula, we have that 
    \begin{equation}\nonumber
       \begin{split}
            K_{\mts{X}}&=(K_{\mb{P}^1\times \mb{P}^2\times U}+\mts{X})|_{\mts{X}}\\
            &=p_1^{*}\mtc{O}_{\mb{P}^1\times \mb{P}^2}(-1,-1)|_{\mts{X}}\otimes p_2^{*}(K_{U}+\pi^{*}\mtc{O}_{\mb{P}^{11}}(1))|_{\mts{X}},
        \end{split} 
    \end{equation}
    and thus $$K_{\mts{X}/U}=p_1^{*}\mtc{O}_{\mb{P}^1\times \mb{P}^2}(-1,-1)|_{\mts{X}}\otimes f^{*}\pi^{*}\mtc{O}_{\mb{P}^{11}}(1)|_{\mts{X}}.$$ Adding the class of $\mts{D}$, we obtain that $$-K_{\mts{X}/U}-c\mts{D}=(1-2c)p_1^{*}\mtc{O}_{\mb{P}^1\times \mb{P}^2}(1,1)|_{\mts{X}}\otimes f^{*}(\pi^{*}\mtc{O}_{\mb{P}^{11}}(-1)-c\mtc{O}_{U}(1)).$$  For simplicity, we denote $p_1^{*}\mtc{O}_{\mb{P}^1\times \mb{P}^2}(a,b)|_{\mts{X}}$ by $\mtc{O}_{\mts{X}}(a,b)$. It is easy to see that $$(\mtc{O}_{\mb{P}^1\times\mb{P}^2}(1,1))^3=3,\quad \textup{and}\quad \mtc{O}_{\mts{X}}(1,1)^2=5.$$ Now we compute that 
    \begin{equation}\nonumber
        \begin{split}
            -f_{*}(-K_{\mts{X}/U}-c\mts{D})^3&=-f_{*}((1-2c)^3\mtc{O}_{\mts{X}}(1,1)^3\\
            &\quad\quad\quad  -3(1-2c)^3\mtc{O}_{\mts{X}}(1,1)^2f^{*}(\pi^{*}\mtc{O}_{\mb{P}^{11}}(-1)-c\mtc{O}_{U}(1))\\
            &\quad\quad\quad  +3(1-2c)\mtc{O}_{\mts{X}(1,1)}f^{*}(\pi^{*}\mtc{O}_{\mb{P}^{11}}(1)+c\mtc{O}_{U}(1))^2\\
            &\quad\quad\quad -f^{*}(\pi^{*}\mtc{O}_{\mb{P}^{11}}(1)+c\mtc{O}_{U}(1))^3)\\
            &=-3(1-2c)^3\pi^{*}\mtc{O}_{\mb{P}^{11}}(1)+3(1-2c)^2\cdot 5(\pi^{*}\mtc{O}_{\mb{P}^{11}}(1)+c\mtc{O}_{U}(1))\\
            &=3(1-2c)^2\left(5(\eta+c\xi)-(1-2c)\eta\right)\\
            &=3(1-2c)^2(4+2c)\left(\eta+\frac{5c}{4+2c}\xi\right).
        \end{split}
    \end{equation}
   
\end{proof}

\begin{corollary}
    For any rational number $0<c<\frac{11}{52}$, we set $$t=t(c):=\frac{5c}{4+2c}.$$ Then on the open subset $U$ of $\mb{P}\mtc{E}$, the CM $\mb{Q}$-line bundle $\Lambda_{\CM,f,c}$ and the polarization $\mtc{L}_t:=\eta+t\xi$ are proportional.
\end{corollary}

\subsection{VGIT wall-crossings and K-moduli}

\begin{theorem}\label{13}
    The VGIT walls are $$t_0=0,\quad t_1=\frac{1}{14},\quad t_2=\frac{1}{8},\quad t_3=\frac{1}{6},\quad t_4=\frac{1}{5},\quad t_5=\frac{1}{4},\quad t_6=\frac{1}{3},\quad t_7=\frac{1}{2},$$ which bijectively correspond via the relation $\displaystyle t(c)=\frac{5c}{4+2c}$ to the K-moduli walls $$c_0=0,\quad c_1=\frac{1}{17},\quad c_2=\frac{2}{19},\quad c_3=\frac{1}{7},\quad c_4=\frac{4}{23},\quad c_5=\frac{2}{9},\quad c_6=\frac{4}{13},\quad c_7=\frac{1}{2}.$$ Among the walls $t_i\in(0,1/2)$, only the one $t_1=\frac{1}{14}$ is a divisorial contraction, and the remaining 5 walls are flips.
\end{theorem}

\begin{proof}
    By the local structure of VGIT, new pairs appear in the GIT moduli stacks $\mtc{M}^{\GIT}(t)$ for each wall $t$. By Lemma \ref{11}, Proposition \ref{12} and discussions of stability thresholds in Section 3 about all the pairs that appear in $\mtc{M}^{\GIT}(t)$ for some $0<t<1/2$, we conclude that $t_0,...,t_7$ listed above are all the walls for the VGIT quotient $\mb{P}\mtc{E}\sslash_t\SL(2)\times \SL(3)$.

    Crossing the first wall, we replace $(\Sigma,D_0)$ by $(X,D)$, where $D_0\in |-2K_{\Sigma}|$ is the curve with a component of multiplicity $4$, $X$ is a quintic del Pezzo surface with exactly one $A_1$-singularity, and $D\in |2K_X|$ does not pass through the singularity of $X$. Counting dimensions, we see that the wall-crossing map $$\ove{M}^{\GIT}\left(\frac{1}{14}+\varepsilon\right)\longrightarrow \ove{M}^{\GIT}\left(\frac{1}{14}\right)$$ is a divisorial contraction. Similarly, we can prove that the other walls are all flips by computing the dimension of the exceptional locus of $\ove{M}^{\GIT}\left(t_i+\varepsilon\right)\rightarrow \ove{M}^{\GIT}\left(t_i\right)$.
\end{proof}

Recall that $U\subseteq \mb{P}\mtc{E}$ is the open subset consisting of complete intersections $(X,D)$. Set $U^K(c)$ to be the subset of $U$ parameterizing pairs $(X,D)$ such that $(X,cD)$ is K-semistable, and $U^{\GIT}(c)$ to be the subset of $U$ parameterizing $\GIT_{t(c)}$-semistable pairs $(X,D)$.

\begin{theorem}\label{12}
    For any $c\in(0,1/17)$, we have $U^K(c)= U^{\GIT}(c)$. Moreover, we have a finite map $$\ove{M}^{\GIT}(c)\longrightarrow \ove{M}^{K}(c)$$ of degree $5$. In fact, we have that $$\ove{M}^{\GIT}(c)\simeq |-2K_{\Sigma}|\sslash \mtf{S}_4,\quad \textup{and}\quad \ove{M}^{K}(c)\simeq |-2K_{\Sigma}|\sslash \mtf{S}_5.$$ 
\end{theorem}

\begin{proof}
    We may assume that $0<c\ll1$. Recall that $U^K(c)\simeq \pi^{-1}(V)$, where $V$ is the open subset of $\mb{P}^{11}$ consisting of all smooth surfaces, and $\pi:\mb{P}\mtc{E}\rightarrow \mb{P}^{11}$ the projection. From Proposition \ref{2}, we get that $U^{\GIT}(c)\subseteq U^K(c)$. As a consequence, one obtains a morphism $$\ove{M}^{\GIT}(c)\longrightarrow \ove{M}^{K}(c).$$ As $\ove{M}^K(c)$ is irreducible and both of the moduli spaces are proper, then this map is surjective. In particular, the c-K-stability and t(c)-GIT-stability are equivalent on $U$ for $0<c<c_1$. 

    Notice that $\SL(2)\times \SL(3)$ acts transitively on the open subset of $\mb{P}^{11}$ consisting of smooth surfaces, and for any point $[X]$ in this locus, the stabilizer of $[X]$ is isomorphic to $\mtf{S}_4$. In fact, viewing $X$ as a pencil of conics $uf_2(x,y,z)+vg_2(x,y,z)=0$ in $\mb{P}^2$, the stabilizer of $X$ consists of $g\in \SL(3)$ fixing the four distinct points $$\{(x:y:z)|f_2(x,y,z)=g_2(x,y,z)=0\}.$$ It follows that $$\ove{M}^{\GIT}(c)\simeq |-2K_{\Sigma}|\sslash \mtf{S}_4.$$ The last statement is exactly \cite[Theorem 1.1]{zha22}. From this, one can also deduce the morphism $\ove{M}^{\GIT}(c)\rightarrow \ove{M}^{K}(c)$ is finite of degree $5$. 
\end{proof}

\section{Hassett-Keel program for curves of genus six}

Hassett-Keel program aims to understand log minimal models of moduli space $\ove{M}_g$ of curves. For any rational number $\alpha\in[0,1]$ such that $K_{\ove{M}_g}+\alpha \Delta$ is big, we can define $$\ove{M}_g(\alpha):=\Proj \bigoplus_{n\geq0}n\left(K_{\ove{M}_g}+\alpha \Delta\right).$$ A natural question is to look for modular interpretation of $\ove{M}_g(\alpha)$. We refer the reader to \cite{HM06} for results on moduli of curves and the survey \cite{FS13} for the Hassett-Keel program. 

In this section, we will prove Theorem \ref{16}. Let us fix some notation first.

\begin{notation}\textup{
Let $U\subseteq \mb{P}\mtc{E}$ be the open subset parameterizing complete intersections $(X,D)$, where $X$ is irreducible, and $V:=\ove{M}_6\setminus(\Delta_1\cup\Delta_2\cup\Delta_3)$ be an open subset in $\ove{M}_6$. Let $0<c<\frac{11}{52}$ be a rational number and $\ove{M}_K(c)$ be the K-moduli space of K-polystable pairs $(X,cD)$ which admits a $\mb{Q}$-Gorenstein smoothable degeneration to $(\Sigma_5,cD_0)$, where $D_0$ is a smooth curve of  class $-2K_{\Sigma_5}$. Let $\phi:U\dashrightarrow \ove{M}_K(c)$ be the natural rational map, $\psi:U\dashrightarrow V$ be given by $(X,D)\mapsto D$, and $\varphi:V\rightarrow \ove{M}^K(c)$ be $C\mapsto (X,C)$, where $X$ is the unique quintic del Pezzo surface containing $C$. Denote by $\ove{\varphi}$ the rational map $\ove{M}_6\dashrightarrow \ove{M}^K(c)$, which is an extension of $\varphi$.}
\end{notation}

\subsection{The first chamber of $\alpha$}

Recall that the K-moduli for $0<c\leq 1/17$ is identified with one of the Hassett-Keel models. 

\begin{theorem} \textup{(ref. \cite[Proposition 4.3]{Mul14})}
Let $0<c\leq \frac{1}{17}$ be a rational number. Then the log canonical model $\ove{M}_6(\alpha)$ is isomorphic to $|-2K_{\Sigma_5}|/\Aut(\Sigma_5)$ for $16/47<\alpha\leq 35/102$, a point for $\alpha=16/47$, and empty for $\alpha<16/47$.
\end{theorem}

\begin{prop}\label{3}
    Neither one of $\Delta_1,\Delta_2,\Delta_3$ is contracted under the natural birational map $\ove{M}_6\dashrightarrow |-2K_{\Sigma_5}|/\Aut(\Sigma_5)$ to the point representing the curve $4L_1+2L_2+2E_1+2E_2$. 
\end{prop}

\begin{proof}
    This is an immediate corollary of \cite[Proposition 2.1, Proposition 2.2, Proposition 2.3]{Mul14}.  
\end{proof}

\subsection{Wall-crossings in the Hassett-Keel program for $\ove{M}_6$}

We may assume that $1/17<c<11/52$.

\begin{prop}
    Let $\lambda$ be the Hodge line bundle on $V$, and $\delta_0$ the divisor on $V$ consisting of the closure of the locus of irreducible curves. Then we have $$\psi^{*}\lambda=6\eta+6\xi,\quad \textup{and}\quad \psi^{*}\delta_0=46\eta+47\xi.$$
\end{prop}

\begin{proof}
    Let $T\simeq \mb{P}^1$ be a curve in $\mb{P}\mtc{E}$ given by a general pencil of curves in $|-2K_X|$ for a fixed smooth surface $X\in \mb{P}^{11}$. Then we have that $$(T.\eta)=0,\quad \textup{and}\quad (T.\xi)=1.$$ The universal family $\mts{C}\rightarrow T$ is the blow-up of $X$ at $(\mtc{O}_X(2,2))^2=20$ general points, and hence $\chi_{\topo}(\mts{C})=27$. By topological Hurwitz formula, we have that the number of singular fibers is $27-2(-10)=47$ since a general fiber is smooth of genus $6$, and each singular fiber is a nodal curve with exactly one singularity. On the other hand, the class of $\mts{C}$ in $T\times X$ is $\mtc{O}(1,-2K_X)$, and hence by adjunction we have $$K_{\mts{C}/T}=\mts{O}_{\mts{C}}(1,-K_{X}).$$ Thus the degree of the Hodge line bundle $\lambda_T$ on $T$ is 
    \begin{equation}\nonumber
        \begin{split}
            \deg \lambda_T&=c_1({p_{1}}_{*}\omega_{\mts{C}/T})\\
            &=c_1(\mtc{O}_{\mb{P}^1}(1)\otimes H^0(X,-K_{X}))\\
            &=6.
        \end{split}
    \end{equation}
    Similarly, fix a general surface $S\in |\mtc{O}_{\mb{P}^1\times\mb{P}^2}(2,2)|$, and take a general pencil of quintic del Pezzo surfaces in $|\mtc{O}_{\mb{P}^1\times\mb{P}^2}(1,2)|$ to intersect $S$, we obtain a one-parameter family of curves $\mts{C}'\rightarrow T'$ over $T'\simeq \mb{P}^1$. Here we view $T'$ as a curve in $\mb{P}\mtc{E}$, and we have that $$(T'.\eta)=1,\quad \textup{and}\quad (T'.\xi)=0.$$ The same computation as above yields that $$(T'.\psi^{*}\delta_0)=46,\quad\textup{and}\quad (T'.\psi^{*}\lambda)=6.$$ It follows immediately that $$\psi^{*}\lambda=6\eta+6\xi,\quad \textup{and}\quad \psi^{*}\delta_0=46\eta+47\xi.$$
\end{proof}

\begin{lemma}\label{4}
 The moduli spaces $\ove{M}_6(\alpha)$ and $\ove{M}^K(c)$ are all normal.
\end{lemma}

\begin{proof}
    Since the deformation of any curve $[C]\in \ove{M}_6(\alpha)$ is unobstructed, there is an \'{e}tale map $$U/\Aut(C)\longrightarrow \ove{M}_6(\alpha)$$ with image an open neighborhood of $[C]$ where $U$ is the first order deformation space of $[C]$. Since $\Aut(X)$ is finite, then $\ove{M}_6(\alpha)$ is normal by a general result of GIT (ref. \cite[Chapter 0]{MFK94}). The normality of K-moduli spaces can be proven as in \cite[Proposition 4.6(3)]{ADL19}.
\end{proof}

\begin{theorem}\label{14}
    Let $0\leq c\leq \frac{11}{52}$ be a rational number, and $\alpha(c)=\frac{32-19c}{94-68c}\in \left[\frac{16}{47},\frac{97}{276}\right]$. Then we have an isomorphism $$\ove{M}_6(\alpha(c))\simeq \ove{M}^K(c).$$
\end{theorem}

\begin{proof}
    Recall that in our setting, there is a commutative diagram 
    $$ \xymatrix{
    U \ar@{-->}[dd]_{\phi} \ar@{-->}[rr]^{\huge{\psi}} && V \ar@{-->}[ddll]^{\varphi}    \\
    \\
    \ove{M}^K(c)  }.$$

    We have that $$\psi^{*}\varphi^{*}\Lambda_{\CM,\ove{M}^K_c,c}=\phi^{*}\Lambda_{\CM,\ove{M}^K_c,c}=\Lambda_{\CM,U,c}=3(1-2c)^2(4+2c)\left(\eta+\frac{5c}{4+2c}\xi\right).$$ On the other hand, we have that $$\psi^{*}(K_{V}+\alpha \delta_0)=\psi^{*}(13\lambda-(2-\alpha)\delta_0)=13(6\eta+6\xi)-(2-\alpha)(46\eta+47\xi),$$ which is proportional, up to a positive constant, to $$\eta+\frac{47\alpha-16}{46\alpha-14}\xi=\eta+\frac{5c}{4+2c}\xi.$$ Thus there is a constant $s>0$ such that $K_{\ove{M}_6}+\alpha\delta-\ove{\varphi}^{*}(s\Lambda_{\CM,\ove{M}^K_c,c})$ is a $\ove{\varphi}$-exceptional $\mb{Q}$-divisor. 
    
    By Proposition \ref{3}, we see that the images of the divisors $\Delta_1,\Delta_2,\Delta_3$ under the rational map $$\phi_{0<c\leq 1/17}:\ove{M}_6\dashrightarrow |-2K_{\Sigma_5}|/\mathfrak{S}_5$$ is not contained in the center of the weighted blow-up $$\ove{M}^K(1/17+\varepsilon)\longrightarrow \ove{M}^K(1/17).$$ 
    Consider the commutative diagram  $$ \xymatrix{
    \ove{M}_6 \ar@{-->}[ddr]_{\ove{\varphi}_{0<c\leq 1/17}} \ar@{-->}[rr]^{\huge{\ove{\varphi}_{1/17<c\leq 11/52}\quad}} && \ove{M}^K(1/17<c\leq 11/52) \ar@{-->}[ddl]^{\sigma}    \\
    \\
    & \ove{M}^K(0<c\leq 1/17)  },$$ where $\sigma$ is the weighted blow-up followed by some birational modifications in codimension $2$. Since any non-special curve of genus 6 is contained in a unique ADE quintic del Pezzo surface, then $\ove{\varphi}_{1/17<c\leq 11/52}$ maps the Gieseker-Petri divisor $\ove{\mtc{GP}}_6$ in $\ove{M}_6$ birationally to the $\sigma$-exceptional divisor $E$. By \cite[Proposition 4.3]{Mul14}, there is a constant $s'>0$ such that $K_{\ove{M}_6}+\alpha\delta-s'\ove{\varphi}^{*}_{0<c\leq 1/17}\Lambda_{\CM,c}$ is $\mb{Q}$-linearly equivalent to $$(35/2-51\alpha)\ove{\mtc{GP}}_6+(9-11\alpha)\delta_1+(19-29\alpha)\delta_2+(34-96\alpha)\delta_3.$$ It follows that $$\left(K_{\ove{M}_6}+\alpha\delta\right)-s\ove{\varphi}^{*}_{\frac{1}{17}<c\leq\frac{11}{52}}\Lambda_{\CM,c}=(9-11\alpha)\delta_1+(19-29\alpha)\delta_2+(34-96\alpha)\delta_3,$$ which is effective when $\alpha\leq \frac{34}{96}$. Therefore by Lemma \ref{4} we deduce
    \begin{equation}\nonumber        \begin{split}\ove{M}_6(\alpha)&=\Proj\bigoplus_{n}H^0\left(\ove{M}_6,n(K_{\ove{M}_6}+\alpha\delta)\right)\\
            &\simeq \Proj\bigoplus_{n}H^0\left(\ove{M}^K(c),ns(\Lambda_{\CM,c})\right)\\
            &=\ove{M}^K(c)
        \end{split}
    \end{equation}
    for $1/17<c\leq 11/52$.
\end{proof}

In \cite{zha22}, we find the walls $$\left\{0,\frac{1}{17},\frac{2}{19},\frac{1}{7},\frac{4}{23}\right\}$$ for K-moduli spaces $\ove{M}^K(c)$. The only wall before $11/52$ that we miss is $1/5$. This will be displayed in our upcoming work \cite{SZ22}. As a consequence, the last walls of the log canonical models $\ove{M}_6(\alpha)$ are $$\left\{\frac{16}{47},\frac{35}{102},\frac{29}{55},\frac{41}{118},\frac{22}{63},\frac{47}{134}\right\}.$$

\begin{remark}
\textup{The value $\alpha(11/52)=329/964$ is not expected to be a wall for the Hassett-Keel. In fact, if it were, then by the local structure of VGIT, there will be divisor appearing in $\ove{M}_6(329/964+\varepsilon)$. However, the Picard group of $\ove{M}_6=\ove{M}_6(1)$ is of rank four, which is generated by $\Delta_0,\Delta_1,\Delta_2,\Delta_3$ and $\ove{\mtc{GP}}_6$. The divisor $\Delta_1,\Delta_2,\Delta_3$ should not be the replacement of triple conics in $|-2K_{\Sigma}|$. In fact, the triple conic is replaced by trigonal curves in the wall-crossing $$\ove{M}^K(11/52+\varepsilon)\rightarrow \ove{M}^K(11/52)\simeq\ove{M}^K(11/52-\varepsilon)$$ for K-moduli, and the locus of trigonal curves in $\ove{M}_6$ is of codimension 2. As a result, the K-moduli spaces cannot characterize the Hassett-Keel for all $\alpha\in[0,1]$.}
\end{remark}

\bibliographystyle{alpha}
\bibliography{citation}

\end{document}